\documentclass[10pt]{amsart}
 \usepackage{amssymb}

 \usepackage{mathrsfs}
\usepackage{amsmath}
\usepackage[all,poly,knot]{xy}
\usepackage{amsthm}
\usepackage{amsfonts}
\usepackage{latexsym}
\usepackage{amscd}
\newtheorem{thm}{Theorem}[section]
\newtheorem{cor}[thm]{Corollary}
\newtheorem{lem}[thm]{Lemma}
\newtheorem{prop}[thm]{Proposition}

\newcommand{\pf}{\noindent\begin {proof}}
\newcommand{\epf}{\end{proof}}

\def\bc{\begin{center}}
\def\ec{\end{center}}

\begin{document}

{\centerline{\large {\bf On Hard Lefschetz Conjecture on Lawson Homology$^{\ast}$}}}

\bigskip
\bigskip

{\centerline{{\large Ze Xu}}}
\bigskip
{\centerline{\it{Institute of Mathematics, Academy of Mathematics and System Science,}}}
{\centerline{\it{Chinese Academy of Sciences, Beijing 100190, China}}}
{\centerline{\it{Email: xuze@amss.ac.cn}}}

\bigskip

{\centerline{{\large January 26, 2011}}}

\bigskip

\begin{figure}[b]

{\small
\begin{tabular}{ll}
$^{\ast}$The author is partially supported by National Natural Foundation
of China (10871106).
\end{tabular}}
\end{figure}

{\bf Abstract:} Friedlander and Mazur proposed a conjecture of hard Lefschetz type on Lawson homology. We shall relate this conjecture to Suslin conjecture on Lawson homology. For abelian varieties, this conjecture is shown to be equivalent to a vanishing conjecture of Beauville type on Lawson homology. For symmetric products of curves, we show that this conjecture amounts to the vanishing conjecture of Beauville type for the Jacobians of the corresponding curves. As a consequence, Suslin conjecture holds for all symmetric products of curves with genus at most 2.

\section{Introduction}

In this note, all varieties are integral schemes of finite type over the complex numbers. If $X$ is a quasi-projective variety of dimension $n$, then the \emph{Lawson homology} of algebraic $p$-cycles on $X$ is defined as
$$L_{p}H_{k}(X):=\pi_{k-2p}(\mathcal{Z}_{p}(X))\ \text{for}\ k\geq 2p \geq 0,$$
where the group of $p$-cycles $\mathcal{Z}_{p}(X)$ is given the Chow topology ([14]). Naturally, one can also define the negative Lawson homology by the same formula in which case $\mathcal{Z}_{p}(X):=\mathcal{Z}_{0}(X\times \mathbb{A}^{-p}):=\mathcal{Z}_{0}(X\times \mathbb{P}^{-p})/\mathcal{Z}_{0}(X\times \mathbb{P}^{-p-1})$ for $p<0$ ([7,11,16]). Like other homology theories, Lawson homology has a cohomological counterpart, \emph{morphic cohomology}, which will be discussed in Section 4. From now on, we assume that Lawson homology and morphic cohomology are all with rational coefficients unless otherwise stated.

Now let $X$ be a smooth projective variety of dimension $n$ and $D$ an ample divisor on $X$. By abuse of notation, $D$ still denotes its class in $L_{n-1}H_{2n-2}(X)\cong NS(X)\otimes_{\mathbb{Z}}\mathbb{Q}$ ([5]). Friedlander and Mazur ([9, 1.3]) proposed  the following hard Lefschetz conjecture on Lawson homology.

\bigskip

{\bf Conjecture 1}\quad  {\it The multiplication map $D^{k-n}\cdot: L_{p}H_{k}(X)\longrightarrow L_{p+n-k}H_{2n-k}(X)$ is injective for $n<k\leq 2n$ (and $k\geq 2p$).}

\bigskip

Inspired by Fu's study on hard Lefschetz conjectures on Chow groups([10]), in this note we study the above conjecture and explain the relation with other conjectures on Lawson homology. Using an observation due to Beilinson([3]), we first show that part of this conjecture is equivalent to Suslin conjecture in Section 2. In Section 3, we prove that for abelian varieties, this conjecture is equivalent to an analogue of Beauville's vanishing conjecture on Chow groups with rational coefficients ([1]). The idea of the proof goes back to [2, Prop.5.11]. In Section 4, we show that for symmetric products of curves, this conjecture is equivalent to the vanishing conjecture of Beauville type for the Jacobians of the corresponding curves by the method in [13]. As a corollary, we get the validity of Suslin conjecture for symmetric products of curves with genus at most 2.

{\bf Acknowledgements}\quad I would like to thank Professor Baohua Fu for encouraging me to write this note, and for the  discussions and suggestions. I thank also Professor Wenchuan Hu for useful comments.
\section{Lawson Homology}

Let us first recall some  basic properties of Lawson homology.

It is known that there is a canonical isomorphism $L_{p}H_{k}(X)\cong H_{k}(X)$ for $p< 0$ ([7,11,16]). Note also that $L_{p}H_{k}(X)=0$, if $p>n$ or $k<2p$. For latter use, denote $LH(X):=L_{*}H_{*}(X)=:\bigoplus_{k\geq 2p,p\in \mathbb{Z}}L_{p}H_{k}(X).$ By a theorem of Friedlander ([5, Th.4.6]), $L_{p}H_{2p}(X)=\pi_{0}(\mathcal{Z}_{p}(X))$ is isomorphic to the group of algebraic $p$-cycles modulo algebraic equivalence. In this context, if $X$ is a smooth projective variety of dimension $n$ and $p, q\in \mathbb{Z}$, then there is an intersection pairing
$$\bullet: L_{p}H_{k}(X)\otimes L_{q}H_{l}(X)\rightarrow L_{p+q-n}H_{k+l-2n}(X)$$
which at the level of $\pi_{0}$ induces the usual intersection pairing on algebraic cycles modulo algebraic equivalence ([6]). This intersection product is graded-commutative and associative ([11, Prop.2.4]). There exists a canonical cycle map $\Phi_{p,k}: L_{p}H_{k}(X)\rightarrow H_{k}(X)$ which is compatible with pull back, push forward and intersection product. In general, it is very hard to compute Lawson homology explicitly. Only a few of cases are known.

Now for two smooth projective varieties $X, Y$ of dimension $n, m$ respectively, denote by $\text{Corr}_{d}(X,Y):=\text{CH}_{n+d}(X\times Y)$ the Chow group of algebraic cycles of dimension $n+d$ on $X\times Y$. By [11], each $\Gamma\in \text{Corr}_{d}(X\times Y)$ determines a homomorphism of Lawson homology $\Gamma_{*}: L_{p}H_{k}(X)\rightarrow L_{p+d}H_{k+2d}(Y)$ for $k\geq 2p$ and $p\in \mathbb{Z}.$ This will play an essential role in this note.

Turning to Conjecture 1, we have the following easy proposition.
\begin{prop}\label{prop1}
 Conjecture 1 holds for $p\leq 0$ and for $p\geq n-1$. In particular, Conjecture 1 holds for all smooth projective varieties of dimension at most 2.
\end{prop}
\begin{proof}  This comes directly from the hard Lefschetz theorem for singular homology and the following known results
on Lawson homology:
if $p<0$, then $L_{p}H_{k}(X)\cong H_{k}(X)$ ([7,11,16]). If $p=0,$ $L_{p}H_{k}(X)=H_{k}(X).$ By [5, Th.4.6], $L_{n-1}H_{k}(X)\cong\mathbb{Q}$ (resp. $H_{2n-1}(X,\mathbb{Q}), NS(X)\otimes_{\mathbb{Z}}\mathbb{Q}$ and 0) if $k=2n$ (resp. $2n-1, 2n-2$ and $k>2n$). If $p=n$, the only nontrivial Lawson homology is $L_{n}H_{2n}(X))\cong \mathbb{Q}$ ([5]).
\end{proof}

Suslin conjecture on Lawson homology (see for example [7]) asserts that for a smooth projective variety $X$ of dimension $n$, the cycle map $\Phi_{p,k}: L_{p}H_{k}(X)\rightarrow H_{k}(X)$ is bijective if $k\geq p+n$. Note that the above statement is a weak form of the original form of Suslin conjecture ([7, 7.9]). In relation with Conjecture 1, we have the following

\begin{prop}\label{prop2}For a smooth projective variety $X$ of dimension $n$, Conjecture 1 implies Suslin conjecture. Conversely, Suslin conjecture implies Conjecture 1 for $k\geq p+n.$
\end{prop}
\begin{proof}Assume Conjecture 1 holds. We have the following commutative diagram
$$\xymatrix{
  L_{p}H_{k}(X) \ar[d]_{\Phi_{p,k}} \ar[rr]^{D^{k-n}\cdot} && L_{p+n-k}H_{2n-k}(X) \ar[d]^{\Phi_{p+n-k,2n-k}} \\
  H_{k}(X) \ar[rr]^{D^{k-n}\cdot } && H_{2n-k}(X).}
$$
The bottom arrow is an isomorphism by hard Lefschetz theorem for singular homology. Note that the cycle map $\Phi_{p,k}$ is an isomorphism if $p<0$. Now it is easy to see the injectivity of $\Phi_{p,k}$ for $k\geq p+n$. Since Suslin conjecture with finite coefficients is true by Milnor-Bloch-Kato conjecture (for a proof, see ([3, p.5]), Suslin conjecture with $\mathbb{Z}$-coefficients is equivalent to the assertion that the Lawson homology $L_{p}H_{k}(X,\mathbb{Z})$ is finitely generated for $k\geq p+n$. Therefore, Suslin conjecture with $\mathbb{Z}$-coefficients amounts to that with $\mathbb{Q}$-coefficients, which is true by the injectivity of the cycle map $\Phi_{p,k}$ for $k\geq p+n$.

The second statement follows easily from the above commutative diagram.
\end{proof}

\section{Abelian Varieties}

Now let $A$ be an abelian variety of dimension $g$ and $\ell\in \text{CH}^{1}(A\times \widehat{A})$ the class of the Poincar\'{e} bundle. Similarly to [1], we have the Fourier transform on Lawson homology $\mathscr{F}=e^{\ell}_{*}: L_{p}H_{k}(A)\rightarrow \bigoplus_{i=p}^{[g+\frac{k}{2}]} L_{p+g-i}H_{k+2g-2i}(\widehat{A})$. For $s\in \mathbb{Z}$, define $L_{p}H_{k}(A)_{(s)}=\{\alpha\in L_{p}H_{k}(A)|m^{*}\alpha=m^{2g-k-s}\alpha\ \text{for all}\ m\in\mathbb{Z}\}$. Then $L_{p}H_{k}(A)=\bigoplus _{s=p-k}^{[g-\frac{k}{2}]}L_{p}H_{k}(A)_{(s)}$ ([12]). Similarly to the situation of Chow groups, $L_{p}H_{k}(A)_{(s)}=(\pi_{2g-k-s})_{*}L_{p}H_{k}(A)_{(s)}$ where $\pi_{i}$ are the canonical Chow-K\"{u}nneth projectors ([4]) of $A$, since $\pi_{i}|_{L_{p}H_{k}(A)_{(s)}}$ is equal to the identity if $i=2g-k-s$ and  to 0 otherwise ([15]). Note that we have an intersection pairing $\bullet:¡¡L_{p}H_{k}(A)_{(s)}\otimes L_{q}H_{l}(A)_{(t)}\rightarrow L_{p+q-g}H_{k+l-2g}(A)_{(s+t)}$. The Fourier transform gives an isomorphism
$$\mathscr{F}: L_{p}H_{k}(A)_{(s)}\simeq L_{p+g-k-s}H_{2g-k-2s}(\widehat{A})_{(s)}$$ if $p-k\leq s\leq [g-\frac{k}{2}].$

In [12], Hu proposed the following conjecture which is an analogue of Beauville's vanishing conjecture on Chow groups ([1]).

\bigskip

{\bf Conjecture 2}\quad {\it For an abelian variety $A$, if $s<0,$ then $L_{p}H_{k}(A)_{(s)}=0$.}

\bigskip

Immediately, we have the following

\bigskip

{\bf Corollary 1}\quad {\it For an abelian variety $A$, Conjecture 2 implies Friedlander-Mazur conjecture ([9, Chapter 7] or [16, 2.7]), i.e., $L_{p}H_{k}(A)=0$ for all $k>2g.$}

\bigskip

We will prove that Conjecture 1 and Conjecture 2 are equivalent for abelian varieties following Beauville's idea ([2]).
First, let us fix some notation. Let $A$ be an abelian variety of dimension $g$ with a polarization $\theta$ of degree $d$. The Pontryagin product $*: L_{p}H_{k}(A)_{(s)}\times L_{q}H_{m}(A)_{(t)}\rightarrow L_{p+q-g}H_{k+m-2g}(A)_{(s+t)}$ on $LH(A)$ is defined by the formula $\alpha*\beta=\mu_{*}(p_{1}^{*}\alpha\cdot p_{2}^{*}\beta)$, where $\mu$ is the multiplication law of $A$ and $p_{i}$ the projection to the $i$-th factor of $A\times A$.

By [2], there exists a morphism of $\mathbb{Q}$-groups $\varphi: \textbf{SL}_{2}\rightarrow \textbf{\textbf{Corr}}^{*}(A)$. The homomorphism $\text{Corr}(A,A)\rightarrow \text{End}_{\mathbb{Q}}(LH(A))$ defines a morphism of $\mathbb{Q}$-groups $\text{\textbf{Corr}}^{*}(A)\rightarrow \text{\textbf{GL}}(LH(A))$. This gives a representation of $\textbf{SL}_{2}$ on $LH(A)$. In sum, similarly to [2, Th. 4.2], we get the following

\bigskip

\begin{lem} There is a representation of $\textbf{SL}_{2}$ on $LH(X),$ which is a sum of finite dimensional representation such that for $n\in\mathbb{Z}-\{0\}, t\in \mathbb{Q}, \alpha\in LH(A),$
$$\left(
  \begin{array}{ccc}
    n & 0 \\
    0 & n^{-1} \\
  \end{array}
\right) \cdot \alpha=n^{-g}n^{*}\alpha, \
\left(
  \begin{array}{ccc}
    0 & -1 \\
    1 & 0 \\
  \end{array}
\right) \cdot \alpha=\mathcal{F}(\alpha)
$$

$$\left(
  \begin{array}{ccc}
    1 & t \\
    0 & 1 \\
  \end{array}
\right) \cdot \alpha=e^{t\theta}\alpha,
\left(
  \begin{array}{ccc}
    1 & 0 \\
    t & 1 \\
  \end{array}
\right) \cdot \alpha=d^{-1}t^{g}e^{\theta/t}*\alpha.
$$

The corresponding action of the Lie algebra $\mathfrak{sl}_{2}(\mathbb{Q})$, in the standard basis

$(X=\left(
  \begin{array}{ccc}
    0 & 1 \\
    0 & 0 \\
  \end{array}
\right),Y=\left(
  \begin{array}{ccc}
    0 & 0 \\
    -1 & 0 \\
  \end{array}
\right),H=\left(
  \begin{array}{ccc}
    1 & 0 \\
    0 & -1 \\
  \end{array}
\right))$, is given by
$$X\alpha=\theta\alpha,\ Y\alpha=d^{-1}\frac{\theta^{g-1}}{(g-1)!}*\alpha,
$$

$$H\alpha=(g-k-s)\alpha,\ \text{for}\ \alpha\in L_{p}H_{k}(X)_{(s)}.
$$
\end{lem}
The interesting point for us is that $H\alpha=\sum_{i}(i-g)\pi_{i*}\alpha=(g-k-s)\alpha$ for $\alpha\in L_{p}H_{k}(A)_{(s)}$.

An explicit description of the action of $\textbf{SL}_{2}$ on $LH(A)$ is in order. An element $\alpha\in L_{p}H_{k}(A)_{(s)}$ is called \emph{primitive} if $\theta^{g-1}*\alpha=0.$ The primitive elements are just the lowest weight elements for the action of $\textbf{SL}_{2}$ on $LH(A).$ If $\alpha\in L_{p}H_{k}(A)_{(s)},$ then the subspace $\sum_{j}\mathbb{Q}\cdot\theta^{j}\alpha$ of $LH(A)$ is an irreducible representation of $\textbf{SL}_{2}$. This space can be identified with the linear space of polynomials in one variable of degree $\leq k+s-g$ with the standard action by the map $f\mapsto f(\theta)\alpha$.  In particular, one has the following

\begin{prop}\label{prop3} If $\alpha\in L_{p}H_{k}(A)_{(s)}$ is primitive, then $k+s-g\geq 0$. The set $\{\alpha,\theta\alpha,\ldots,\theta^{k+s-g}\alpha\}$ forms a basis of an irreducible subrepresentation of $LH(A)$ and  $LH(A)$ is a direct sum of subrepresentations of this type. If $P_{(s)}^{p,k}\subseteq L_{p}H_{k}(A)_{(s)}$ denotes the subspace of primitive elements, then we get $$L_{p}H_{k}(A)_{(s)}=\bigoplus_{j=0}^{g-p}\theta^{j}P_{(s)}^{p+j,k+2j}.$$
\end{prop}

Now we can prove the main theorem in this section.
\begin{thm} For an abelian variety $A$, Conjecture 1 is equivalent to Conjecture 2.
\end{thm}
\begin{proof} First, suppose Conjecture 2 holds. An ample divisor $D\in L_{g-1}H_{2g-2}(A)$ has a decomposition $D=\theta+\theta_{1}$ where $\theta=\frac{D+(-1)^{*}D}{2}\in L_{g-1}H_{2g-2}(A)_{(0)}$ and $\theta_{1}=\frac{D-(-1)^{*}D}{2}\in L_{g-1}H_{2g-2}(A)_{(1)}$. Then $\theta$ is ample and symmetric. Therefore, $A$ can be considered as an abelian variety with the polarization $\theta$. Since the injectivity of $\theta^{k-g}\cdot:L_{p}H_{k}(A)\rightarrow L_{p+g-k}H_{2g-k}(A)$ for $g<k\leq 2g$ implies the injectivity of $D^{k-g}\cdot: L_{p}H_{k}(A)\rightarrow L_{p-g-k}H_{2g-k}(A)$ for $g<k\leq 2g$, it suffices to show that if $s\geq 0$, then the multiplication map $\theta^{k-g}\cdot: L_{p}H_{k}(A)_{(s)}\rightarrow L_{p-j}H_{k-2j}(A)_{(s)}$ is injective. Now assume that $\alpha\in L_{p}H_{k}(A)_{(s)}$ such that $\theta^{k-g}\alpha=0.$ By Proposition \ref{prop3}, $\alpha=\sum_{r=0}^{g-p}\theta^{r}\alpha_{p+r}$ with $\alpha_{p+r}\in P_{(s)}^{p+r,k+2r}.$ Then $\theta^{k-g+r}\alpha_{p+r}=0$ for each $0\leq r\leq g-p.$ Since $k-g+r\leq k+s-g+r\leq (k+2r)+s-g,$ by Proposition \ref{prop3} again, $\alpha_{p+r}=0$ for each $0\leq r\leq g-p.$ Hence, $\alpha=0.$

Now assume that Conjecture 1 holds. Suppose that there is a nonzero primitive element $\alpha\in L_{p}H_{k}(A)_{(s)}$ for some $s<0.$ Let $\theta$ be the component of $D$ in $L_{g-1}H_{2g-2}(A)_{(0)}.$ Since $k-g>k+s-g,$ by Proposition \ref{prop3}, $\theta^{k-g}\alpha=0$. Note that $D=\textbf{T}_{a}^{*}\theta$ for some $a\in A,$ where $\textbf{T}_{a}$ is the translation of $A$ by $a$. Since $(\textbf{T}_{a}-\text{id})^{*}\alpha=\alpha*([-a]-[o])\in L_{p}H_{k}(A)_{(s+1)},$ we must have $\textbf{T}_{a}^{*}\alpha\neq 0$. But $D^{k-g}\cdot \textbf{T}_{a}^{*}\alpha=0,$ a contradiction.
\end{proof}
\section{Symmetric Products of Curves}

For the convenience of statement, in this section we will mainly consider morphic cohomology, which is the cohomological counterpart of Lawson homology and it has similar properties as those of Lawson homology ([8]). Furthermore, there is a canonical duality homomorphism $\mathcal{D}: L^{q}H^{l}(X)\rightarrow L_{n-q}H_{2n-l}(X)$ which is compatible with the cycle maps. For smooth projective varieties, the duality homomorphism is an isomorphism ([8]) and $L^{*}H^{*}(X):=\bigoplus_{l\leq 2q,q\in \mathbb{Z}}L^{q}H^{l}(X)$ is a bigraded ring with the intersection product as multiplication and being anti-commutative for the second grading ([11, Prop.2.4]). Since we are always interested in smooth projective varieties, every statement about Lawson homology has an equivalent dual statement about morphic cohomology.

Now let $C$ be a smooth projective curve with genus $g$ and $P_{0}\in C$ a fixed point. The $n$-th symmetric product of $C$ is denoted by $C^{(n)}.$  The Jacobian of $C$ is denoted by $J:=J(C).$ The morphism $\phi_{n}: C^{(n-1)}\rightarrow C^{(n)}$ is the addition of the point $P_{0}.$ Note that $C^{(0)}=\{P_{0}\}$ and $\phi_{0}$ is the inclusion of the point $P_{0}$ in $C$. We will need the following well-known fact ([13, Prop.2.7]). Let $\mathcal{P}$ be the Poincar\'{e} line bundle on $J\times C$, $p: J\times C\rightarrow J$ and $q: J\times C\rightarrow C$ the projections. Define $\mathcal{E}_{n}=p_{*}(\mathcal{E}_{n}\otimes q^{*}\mathcal{O}(nP_{0})).$ Then the projectivization $\mathbb{P}(\mathcal{E}_{n})$ is canonically isomorphic to $C^{(n)}$ and the natural homomorphism $\mathcal{E}_{n-1}\rightarrow \mathcal{E}_{n}$ induces the morphism $\phi_{n}.$ Therefore, $\phi_{n}^{*}\mathcal{O}_{\mathcal{E}_{n}}(1)=\mathcal{O}_{\mathcal{E}_{n-1}}(1)$. If $n>2g-2,$ $\mathcal{E}_{n}$ is a locally free sheaf, $C^{(n)}$ is a projective bundle over $J$ via the Albanese map and there is an exact sequence of sheaves on $J$: $0\rightarrow \mathcal{E}_{n}\rightarrow \mathcal{E}_{n+1}\rightarrow \mathcal{O}_{J}\rightarrow 0.$ Moreover, the divisor $\phi_{n*}([C^{(n)}])\in CH^{1}(C^{(n)})$ is the first Chern class of the line bundle $\mathcal{O}_{\mathcal{E}_{n}}(1).$ Now define the \emph{infinite symmetric product} of $C$, denoted by $C^{(\infty)}$, to be the direct system $(C^{(n)}, \phi_{n})$ ([13]). Then we can regard all the $\mathcal{O}_{\mathcal{E}_{n}}(1)$ as a line bundle $\mathcal{O}(1)$ on $C^{(\infty)}$ (via $\phi_{n}^{*}$). The class of the divisor $\phi_{n}(C^{(n-1)})\in CH^{1}(C^{(n)})$ is denoted by $z_{n}$ which will also be used to denote its image in $L^{1}H^{2}(C^{(n)}).$ Define the correspondence $\Psi_{n}: C^{(n)}\vdash C^{(n-1)}$ to be $\sum_{i=0}^{n-1}\sum_{1\leq a_{1}<\ldots<a_{i}\leq n}(-1)^{n-1-i}(P_{0},...,P_{0},\pi_{a_{1}},\ldots,\pi_{a_{i}})$. Then as correspondences, $\Psi_{n}\circ \phi_{n*}=\Delta_{C^{(n-1)}}$ ([13]). Transposing the above identity, we get $\phi_{n}^{*}\circ (^{t}\Psi_{n})=\Delta_{C^{(n-1)}}.$ Hence, $\phi_{n*}: L^{*}H^{*}(C^{(n-1)})\rightarrow L^{*}H^{*}(C^{(n)})$ is injective, while $\phi_{n}^{*}: L^{*}H^{*}(C^{(n)})\rightarrow L^{*}H^{*}(C^{(n-1)})$ is surjective. Define $L^{*}H^{*}(C^{(\infty)}):=\varprojlim_{n}L^{*}H^{*}(C^{(n)})$ via $\phi_{n}^{*}$ and  $L^{*}H^{*}(C^{(+)}):=\varinjlim_{n}L^{*}H^{*}(C^{(n)})$ via $^{t}\Psi_{n*}$. Exactly the same argument as that in [13, Prop.1.17] shows that there is a canonical injective map $L^{*}H^{*}(C^{(+)})\rightarrow L^{*}H^{*}(C^{(\infty)})$ defined by sending $\alpha\in L^{*}H^{*}(C^{(n)})$ to $(\ldots,(\phi_{n}^{2})^{*}\alpha,\phi_{n}^{*}\alpha,\alpha,(^{t}\Psi_{n})_{*}\alpha,(^{t}\Psi_{n})_{*}^{2}\alpha,\ldots)$, which makes $L^{*}H^{*}(C^{(+)})$ into a subring of $L^{*}H^{*}(C^{(\infty)})$. Then we have the following
 \begin{lem}\label{lem2}
 As bigraded rings, $L^{*}H^{*}(C^{(\infty)})\cong L^{*}H^{*}(J)[H]$ where $H$ is the image of the Chern class of $\mathcal{O}(1)$ $c_{1}(\mathcal{O}(1))$ in $L^{1}H^{2}(C^{(\infty)})$.
 \end{lem}
\begin{proof}First, note that $c_{1}(\mathcal{O}(1))\in CH^{1}(C^{(\infty)})$ makes sense. From the above discussion, if $n>2g-2,$ the Albanese map $C^{(n)}\rightarrow J$ makes $C^{(n)}$ into a projective bundle over $J$. By [5, Prop.2.5], it is easy to see that there is a canonical bigraded ring isomorphism $L^{*}H^{*}(C^{(n)})\cong \frac{L^{*}H^{*}(J)[H]}{(F_{n}(H))}$, where $H$ is the image of $c_{1}(\mathcal{O}(1))$ in $L^{1}H^{2}(C^{(\infty)})$ and $F_{n}$ is a monic polynomial of degree $n-g+1.$ Assume now $n>2g-1.$ Then $\phi_{n}^{*}: L^{q}H^{l}(C^{(n)})\rightarrow L^{q}H^{l}(C^{(n-1)})$ is bijective for $q<(n-1)-g+1=n-g.$ If $q\geq g-1,$ then $L^{q}H^{l}(C^{(\infty)})\cong L^{q}H^{l}(C^{(q+g)})$  which is isomorphic to the $(q,l)$-part of $L^{*}H^{*}(J)[H].$ If $q<g-1,$ then $L^{q}H^{l}(C^{(\infty)})\cong L^{q}H^{l}(C^{(2g-1)})$ which is also isomorphic to the $(q,l)$-part of $L^{*}H^{*}(J)[H].$ Therefore, $L^{*}H^{*}(C^{(\infty)})\cong L^{*}H^{*}(J)[H].$
\end{proof}

Let $\Gamma_{n,m}:=\sum_{1\leq a_{1},...,a_{n}\leq m}(\pi_{a_{1}},...,\pi_{a_{n}}): C^{m}\rightarrow C^{(n)}$. Then for $\alpha\in LH^{*}(C^{(n)})$ with $\phi_{n}^{*}\alpha=0,$ as in [13, Prop.2.1], $\phi_{m+1}^{*} (^{t}\Gamma_{n,m+1})_{*}\alpha=(^{t}\Gamma_{n,m})_{*}\alpha.$ Then we can define a map $K_{n}: \ker\phi_{n}^{*}\rightarrow L^{*}H^{*}(C^{(\infty)})$ by sending $\alpha$ to $((^{t}\Gamma_{n,0})_{*}\alpha,...,(^{t}\Gamma_{n,m})_{*}\alpha,...)$. The same proof as that in [13, Th.2.3] gives the following

\begin{lem}\label{lem3} The homomorphism $K_{n}$ is injective and $L^{*}H^{*}(C^{(+)})=\bigoplus_{n=0}^{\infty}K_{n}(\ker\phi_{n}^{*})$. Denote by $N_{C^{(\infty)}}: C^{(\infty)}\rightarrow C^{(\infty)}$ the multiplication by $N$ (defined by the $N$-tuple diagonal embedding), then $\ker\phi_{n}^{*}$ has eigenvalue $N^{n}$ for the action of $N_{C^{(\infty)}}^{*}$.\end{lem}

Note that for morphic cohomology, if $A$ is an abelian variety of dimension $g$, then $L^{q}H^{l}(A)=\bigoplus_{s=l-g-q}^{[\frac{l}{2}]}L^{q}H^{l}(A)_{(s)}.$ The cohomological version of Conjecture 2 is that if $s<0,$ then $L^{q}H^{l}(A)_{(s)}=0.$

Now we can prove the main theorem of this section.

\begin{thm}\label{thm2} Given a smooth projective curve $C$ with genus $g$, the following three statements are equivalent:

(i) For any integer $n\geq 1$, Conjecture 1 holds for $(C^{(n)}, z_{n})$.

(ii) For any integer $n\geq 1,$ the homomorphism $\phi_{n}^{*}: L^{q}H^{l}(C^{(n)})\rightarrow L^{q}H^{l}(C^{(n-1)})$ is an isomorphism for $0\leq l< n$ and $l\leq 2q.$

(iii) Conjecture 2 holds for $J=J(C)$.
\end{thm}

\begin{proof} (i)$\Leftrightarrow$(ii). Note that the homomorphism $z_{n}\cdot: L^{q}H^{l}(C^{(n)})\rightarrow L^{q}H^{l}(C^{(n)})$ is equal to $\phi_{n}^{*}\circ \phi_{n*}.$ Since $\phi_{n}^{*}$ is surjective and $\phi_{n*}$ is injective, the injectivity of $z_{n}\cdot$ is equivalent to the injectivity of $\phi_{n}^{*}$. Now the equivalence follows.

(iii)$\Rightarrow$(ii). Suppose that Conjecture 2 holds for $J=J(C)$. Sine $\phi_{n}^{*}$ is surjective, it suffices to prove the injectivity. By Lemma \ref{lem2}, $L^{q}H^{l}(C^{(\infty)})=\bigoplus_{j=0}^{q}L^{q-j}H^{l-2j}(J)\cdot H^{j}.$ Note that $H$ has eigenvalue $N$ for the multiplication by $N$ ([13, Lem.2.10]). Then for each $j$, the possible eigenvalues of $H^{l-2j}(J)\cdot H^{j}$ for the multiplication by $N$ are $N^{[\frac{l}{2}]},\ldots,N^{l-j}$. Hence the possible eigenvalues of $L^{q}H^{l}(C^{(\infty)})$ for the multiplication by $N$ are $N^{l-q},\ldots,N^{l}.$ On the other hand, by Lemma \ref{lem3}, $L^{*}H^{*}(C^{(+)})=\bigoplus_{n=0}^{\infty}K_{n}(\ker\phi_{n}^{*}),$ where $K_{n}(\ker\phi_{n}^{*})$ has eigenvalue $N^{n}$ for the multiplication by $N$. Then $K_{n}(\ker\phi_{n}^{*})=0$ if $l<n.$ Since $K_{n}$ is injective, $\phi_{n}^{*}$ is injective if $l<n.$

(ii)$\Rightarrow$(iii). Now assume that for any integer $n\geq 1,$ $\phi_{n}^{*}$ is bijective for $0\leq l<n$ and $l\leq 2q.$ Since $L^{*}H^{*}(J)$ is a subring of $L^{*}H^{*}(C^{(\infty)})$ as bigraded rings, then $L^{q}H^{l}(J)\subseteq L^{q}H^{l}(C^{(\infty)})$. Therefore the possible eigenvalues of $L^{q}H^{l}(J)$ for the multiplication by $N$ are $N^{0},\ldots,N^{l}$. This shows that $L^{q}H^{l}(J)_{(s)}=0$ if $s<0.$
\end{proof}

Combining Proposition \ref{prop1} and Theorem \ref{thm2}, we get the following

\begin{cor}\label{cor2} For any integer $n\geq 1,$ Conjecture 1 holds for $(C^{(n)}, z_{n})$ if $g(C)\leq 2.$
\end{cor}

Now we can obtain the validity of Suslin's conjecture for $C^{(n)}$ if $g(C)\leq 2$, which can also be deduced
from [7].

\begin{cor} Suslin conjecture holds for all symmetric products of curves with genus at most 2, i.e., for any integer $n\geq 1$, the cycle map $\Phi_{p,k}: L_{p}H_{k}(C^{(n)})\rightarrow H_{k}(C^{(n)})$ is bijective for $k\geq p+n$ if $g(C)\leq 2.$
\end{cor}
\begin{proof}This follows from Proposition \ref{prop2}, Corollary \ref{cor2} and the proof of Proposition \ref{prop2}.
\end{proof}
\bigskip

\bigskip

\bigskip

{\bf References}

\bigskip

[1] Beauville A., Quelques remarques sur la transformation de Fourier dans l'anneau de Chow d'une vari\'{e}t\'{e} ab\'{e}lienne. Algebraic geometry (Tokyo/Kyoto, 1982), 238-260, LNM 1016, Springer, Berlin, 1983.

[2] Beauville A., The action of $\textbf{SL}_{2}$ on abelian varieties. J. Ramanujan Math. Soc., 2010, 25: 253-263.

[3] Beilinson A., Remarks on Grothendieck's standard conjectures. arXiv: 1006.1116v2 [math.AG], 2010.

[4] Deninger C. and Murre J. P., Motivic decomposition of abelian schemes and the Fourier transform. J. Reine Angew. Math. 1991, 422: 201-219.

[5] Friedlander E. M., Algebraic cycles, Chow varieties, and Lawson homology. Compositio Math., 1991, 77: 55-93.

[6] Friedlander E. M. and Gabber O., Cycles spaces and intersection theory. In Topological methods in modern mathematics (Stony Brook, NY, 1991), 325-370.

[7] Friedlander E. M., Haesemeyer and Walker M. E., Techniques, computations, and conjectures for semi-topological K-theory. Math. Ann., 2004, 330: 759-807.

[8] Friedlander E. M. and Lawson, H. B., Jr., Duality relating spaces of algebraic cocycles and cycles. Topology, 1997, 36: 533-565.

[9] Friedlander E. M. and Mazur B., Filtrations on the homology of algebraic varieties. Mem. Amer. Math. Soc., \textbf{110}(529):x+110,1994. With an appendix by Daniel Quillen.

[10] Fu B., Remarks on Hard Lefschetz conjectures on Chow groups. Sci. China Math., 2010, 53: 105-114.

[11] Hu W. and Li L., Lawson homology, morphic cohomology and Chow motives. arXiv: 0711.0383v1 [math.AG], 2007.

[12] Hu W., Lawson homology for abelian varieties. Preprint.

[13] Kimura S. and Vistoli A., Chow rings of infinite symmetric products. Duke Math. J., 1996, 85: 411-430.

[14] Lawson, H. B. Jr., Algebraic cycles and homotopy theory. Ann. of Math., 1989, 129: 253-291.

[15] Murre J. P., On a conjectural filtration on the Chow groups of an algebraic variety I. The general conjectures and some examples. Indag. Math. (N. S.), 1993, 4: 177-188.

[16] Voineagu M., Cylindrical homomorphisms and Lawson homology. J. K-theory, 2010: 1-34.

\end{document}